\documentclass{article}
\usepackage{amsmath,amssymb,amsthm,url,graphics}
\usepackage{tikz}

\title{Refinements of provability and consistency principles for the second incompleteness theorem}
\author{Taishi Kurahashi\footnote{Email: kurahashi@people.kobe-u.ac.jp}
\footnote{Graduate School of System Informatics,
Kobe University,
1-1 Rokkodai, Nada, Kobe 657-8501, Japan.}
}
\date{}

\theoremstyle{plain}
\newtheorem{thm}{Theorem}[section]
\newtheorem{lem}[thm]{Lemma}
\newtheorem{prop}[thm]{Proposition}
\newtheorem{cor}[thm]{Corollary}
\newtheorem{fact}[thm]{Fact}
\newtheorem{prob}[thm]{Problem}

\theoremstyle{definition}
\newtheorem{defn}[thm]{Definition}

\newcommand{\PA}{\mathsf{PA}}
\newcommand{\PR}{\mathrm{Pr}}

\newcommand{\Proof}{\mathrm{Proof}}

\newcommand{\ConH}{\mathrm{Con}^{\mathrm{H}}}
\newcommand{\ConS}{\mathrm{Con}^{\mathrm{S}}}
\newcommand{\ConG}{\mathrm{Con}^{\mathrm{G}}}
\newcommand{\ConL}{\mathrm{Con}^{\mathrm{L}}}
\newcommand{\Ros}{\mathbf{Ros}}
\newcommand{\RFN}{\mathrm{RFN}}
\newcommand{\gn}[1]{\ulcorner#1\urcorner}

\newcommand{\D}[1]{\mathbf{D#1}}
\newcommand{\DU}[1]{\mathbf{D#1^{\mathrm{U}}}}
\newcommand{\DG}[1]{\mathbf{D#1^{\mathrm{G}}}}
\newcommand{\M}{\mathbf{M}}
\newcommand{\MU}{\mathbf{M}^{\mathrm{U}}}
\newcommand{\E}{\mathbf{E}}
\newcommand{\EU}{\mathbf{E}^{\mathrm{U}}}
\newcommand{\C}{\mathbf{C}}
\newcommand{\GCU}{\mathbf{\Gamma C}^{\mathrm{U}}}
\newcommand{\SC}{\mathbf{\Sigma_1 C}}
\newcommand{\SCU}{\mathbf{\Sigma_1 C}^{\mathrm{U}}}
\newcommand{\DCU}{\mathbf{\Delta_0 C}^{\mathrm{U}}}

\newcommand{\CB}{\mathbf{CB}}
\newcommand{\CBP}{\mathbf{CB^+}}
\newcommand{\CBE}{\mathbf{CB_{\exists}}}
\newcommand{\CBEP}{\mathbf{CB_{\exists}^+}}
\newcommand{\PC}{\mathbf{PC}}
\newcommand{\PCU}{\mathbf{PC^\mathrm{U}}}
\newcommand{\Fml}{\mathrm{Fml}_{\mathcal{L}_A}}
\newcommand{\num}{\overline}
\newcommand{\True}{\mathrm{True}}
\newcommand{\Q}{\mathsf{Q}}
\newcommand{\LA}{\mathcal{L}_A}
\newcommand{\s}{\mathsf{s}}

\newcommand{\HB}[1]{\mathbf{HB#1}}

\begin{document}

\maketitle

\begin{abstract}
This paper continues the author's previous study \cite{Kura20}, showing that several weak principles inspired by non-normal modal logic suffice to derive various refined forms of the second incompleteness theorem.
Among the main results of the present paper, we show that the set $\{\mathbf{E},\mathbf{C}, \mathbf{D3}\}$ suffices to establish the unprovability of the consistency statement $\neg\, \mathrm{Pr}_T(\ulcorner 0=1 \urcorner)$. 
We also prove that the set $\{\mathbf{E}^{\mathrm{U}}, \mathbf{CB_{\exists}}\}$ yields formalized $\Sigma_1$-completeness. 
\end{abstract}

\section{Introduction}

The second incompleteness theorem (G2), originally established by G\"odel \cite{Goed31}, states that any sufficiently strong and consistent theory $T$ of arithmetic cannot prove its own consistency. 
However, the correct statement of G2 in fact depends heavily on the choice of provability predicates and the formulation of the consistency statement. 
To state G2 precisely, these parameters must be carefully taken into account.
The author’s previous work \cite{Kura20} provided a systematic account of this complex situation underlying G2. 
It distinguished several formulations of consistency such as G\"odel’s $\ConG_T \equiv \exists x \, \bigl(\Fml(x) \land \neg \, \PR_T(x) \bigr)$, Hilbert and Bernays' $\ConH_T \equiv \forall x\, (\Fml(x) \land \PR_T(x) \to \neg \, \PR_T(\dot{\neg} \, x))$ \cite{HB39} (see also Feferman \cite{Fef60}), and L\"ob’s $\ConL_T : \equiv \neg \, \PR_T(\gn{0 = 1})$ \cite{Lob55}.
It also showed that different derivability conditions for provability predicates are required to establish the unprovability of each. 
This revealed that G2 is not a single theorem, but a collection of several theorems stating the unprovability of consistency statements under appropriate derivability conditions.
This observation suggests building a framework that shows the structural differences among various conditions of provability predicates and formulations of consistency principles.

Among the key results of \cite{Kura20} are the following:
\begin{enumerate}
    \item The combination of the global versions of $\D{2}$ and $\PC$ suffices for G2 of $\ConG_T$ (\cite[Corollary 2.8]{Kura20}), but the combination of the uniform version of $\D{1}$ and the global versions of $\D{2}$ and $\SC$ do not (\cite[Proposition 4.11]{Kura20}). 
    For the details of these principles, please see the next section. 
    
    \item For G2 concerning $\ConH_T$, Hilbert and Bernays essentially showed that the set $\{\M, \CB, \DCU \}$ suffices.
    Jeroslow \cite{Jer73} and Kreisel--Takeuti \cite{KT74} pointed out that $\{\SC\}$ is also sufficient.
    In addition, \cite[Theorems 2.8 and 2.10]{Kura20} established that both $\{\M, \D{3}\}$ and $\{\PC\}$ independently suffice.

    \item For $\ConL_T$, it is well known that the set $\{\D{2}, \D{3}\}$ suffices to derive G2.
    The paper showed that none of the above four sets of derivability conditions sufficient for G2 of $\ConH_T$ is sufficient for G2 of $\ConL_T$ (\cite[Fact 4.6.3, Propositions 4.4 and 4.10]{Kura20}).

    \item The paper also investigates various implication relations among provability conditions. 
    In particular, it was proved that $\SCU$ is derived only from $\MU$ (\cite[Theorem 2.20]{Kura20}).
\end{enumerate}

The properties of provability predicates have long been analyzed through the normal modal logic $\mathsf{GL}$ (cf.~\cite{AB05,Boo93,JD98,Sol76}).
On the other hand, recent studies \cite{Kogu,KK23,Kura24} have investigated provability predicates via the framework of non-normal modal logic.
This approach shows that some principles and rules from non-normal modal logic can be useful in studying the second incompleteness theorem. 
For example, in the paper \cite{Kura21}, the author introduced the schematic consistency statement $\ConS_T = \{\PR_T(\gn{\varphi}) \to \neg\, \PR_T(\gn{\neg\, \varphi}) \mid \varphi$ is a formula$\}$ that corresponds in modal logic to the axiom \textsc{D}: $\Box A \to \neg \Box \neg A$. 
It is well known that \textsc{D} is equivalent, in normal modal logic, to the axiom \textsc{P}: $\neg \Box \bot$, which corresponds to the consistency statement $\ConL_T$.
However, in non-normal modal logics weaker than the weakest normal modal logic $\mathsf{K}$, the axioms \textsc{D} and \textsc{P} are not equivalent in general.
Correspondingly, \cite{Kura21} proved the existence of a Rosser provability predicate for which $\ConS_T$ not is provable, thereby the versions of G2 for $\ConL_T$ and $\ConS_T$ are shown to be distinct.

Motivated by these developments, the present paper further investigates principles and rules that have been studied in the context of non-normal modal logic, within the framework of the second incompleteness theorem.
In particular, we focus on the three rules \textsc{RE}: $\dfrac{A \leftrightarrow B}{\Box A \leftrightarrow \Box B}$, \text{RM}: $\dfrac{A \to B}{\Box A \to \Box B}$, and \textsc{Nec}: $\dfrac{A}{\Box A}$, and the axiom \textsc{C}: $\Box A \land \Box B \to \Box (A \land B)$ (cf.~\cite{Che80,Pac17}).
In standard non-normal modal logic, study is typically based on the rule \textsc{RE} from the perspective of neighborhood semantics (cf.~\cite{Pac17}).

By contrast, our analysis is based on the rule \textsc{Nec}, reflecting the perspective of provability predicates.
Indeed, in \cite{Kura24}, it was proved that the non-normal modal logic $\mathsf{N}$, which is obtained from classical propositional logic by adding the rule \textsc{Nec} and was introduced by Fitting et al.~\cite{FMT92}, coincides with the provability logic of all provability predicates. 
The paper \cite{Kura24} also introduced the logic $\mathsf{NR}$ that is obtained from $\mathsf{N}$ by adding the Rosser rule \textsc{Ros}: $\dfrac{\neg A}{\neg \Box A}$ and proved that $\mathsf{NR}$ is exactly the provability logic of all Rosser provability predicates. 
The rule \textsc{Ros} lies between \textsc{P} and \textsc{D} over $\mathsf{N}$, and they are strictly distinguished. 
This suggests that they correspond to different consistency principles in the context of arithmetic. 
The rule $\M$ of arithmetic corresponding to \textsc{RM} has already been introduced as early as in the work of Hilbert and Bernays \cite{HB39}, and its modal logical study was established in \cite{KK23}. 

In the present paper, we introduce and examine the rule $\E$ and the principle $\C$ of arithmetic corresponding to \textsc{RE} and \textsc{C}, respectively.
In addition to the schematic consistency $\ConS_T$, we also introduce a new consistency rule $\Ros$, which corresponds to the Rosser rule \textsc{Ros}, and we analyze G2 pertaining to these formulations.

The main results of the present paper are summarized as follows:

\begin{itemize}
    \item We show that most of the results previously derived using $\M$ can in fact be obtained using the weaker principle $\E$ instead. 
    In particular, the unprovability of $\ConS_T$ follows from $\{\E, \D{3}\}$ (Theorem \ref{G2_refine_E}).

    \item The principle $\C$ ensures the equivalence between $T \vdash \ConS_T$ and the validity of $\Ros$, while $\E$ guarantees the equivalence between the validity of $\Ros$ and $T \vdash \ConL_T$ (Proposition \ref{EC}). 
    Combining these results, we obtain that the unprovability of $\ConL_T$ from $\{\E, \C, \D{3}\}$, which refines the classical result that $\{\D{2}, \D{3}\}$ suffices for the unprovability of $\ConL_T$. 
    We prove that $\{\E, \C, \D{3}\}$ is in fact strictly weaker than $\{\D{2}, \D{3}\}$ (Theorem \ref{ex1}). 

    \item We prove that every $\Sigma_1$ provability predicate satisfying $\C$ and $\D{3}$ does not satisfy $\Ros$ (Theorem \ref{G2_C3}). 
    This a new version of G2 concerning the newly introduced principle $\C$ and consistency rule $\Ros$. 

    \item We prove that $\M$ is redundant for Hilbert and Bernays’ version of G2, that is, the set $\{\CB, \DCU\}$ suffices to derive the unprovability of $\ConH_T$ (Theorem \ref{HB_refine}).
    This is proved by showing that $\CB$ is sufficient to establish the unprovability of the $\Delta_0$ uniform reflection principle $\RFN_T(\Delta_0)$. 

    \item We prove that the set $\{\EU, \CBE\}$ is sufficient for $\SC$ (Theorem \ref{EU_to_SCU}). 
    Thus, $\{\EU, \CBE\}$ suffices for G2 of $\ConS_T$, and $\{\EU, \CBE, \C\}$ suffices for G2 of $\ConL_T$. 
\end{itemize}

Figure \ref{Fig1} summarizes several prominent implications between the provability conditions and the versions of G2. 

\begin{figure}[ht]
\centering
\begin{tikzpicture}
\node (ConG) at (0,-1) {$\nvdash \ConG_T$};
\node (ConL) at (2,-1) {$\nvdash \ConL_T$};
\node (Ros) at (4,-1) {non-$\mathbf{Ros}$};
\node (ConS) at (7,-1) {$\nvdash \ConS_T$};
\node (ConH) at (9,-1) {$\nvdash \ConH_T$};

\node (EC3) at (2,2) {$\{\E,\C,\D{3}\}$};
\node (EC3g) at (2,1) {$\{\E,\C,\D{3}^n_{n+k}\}$};
\node (E3) at (5.7,1) {$\{\E,\D{3}\}$};
\node (S) at (7,1) {$\{\SC\}$};

\node (C3) at (4,1) {$\{\C,\D{3}\}$};
\node (C3g) at (4,0) {$\{\C,\D{3}^n_{n+k}\}$};

\node (EUCBE) at (5.7,2) {$\{\EU,\CBE\}$};

\node (23) at (2,3) {$\{\D{2},\D{3}\}$};

\node (MU) at (5.7,3) {$\{\MU\}$};
\node (1U2U) at (2,5) {$\{\DU{1},\DU{2}\}$};
\node (EUCBEC) at (2,4) {$\{\EU, \CBE, \C\}$};

\node (DCUCB) at (9,1) {$\{\CB, \DCU\}$};
\node (2GPCG) at (0,6) {$\{\DG{2},\PC^{\mathrm{G}}\}$};

\draw [->, double] (ConG)--(ConL);
\draw [->, double] (ConL)--(Ros);
\draw [->, double] (Ros)--(ConS);
\draw [->, double] (ConS)--(ConH);

\draw [->, double] (EC3)--(EC3g);
\draw [->, double] (EC3g)--(C3g);
\draw [->, double] (EC3g)--(ConL);
\draw [->, double] (EC3)--(E3);
\draw [->, double] (E3)--(ConS);
\draw [->, double] (23)--(EC3);
\draw [->, double] (MU)--(EUCBE);
\draw [->, double] (EUCBE)--(E3);
\draw [->, double] (EUCBEC)--(23);
\draw [->, double] (EUCBEC)--(EUCBE);
\draw [->, double] (EUCBE)--(S);
\draw [->, double] (S)--(ConS);
\draw [->, double] (1U2U)--(EUCBEC);
\draw [->, double] (1U2U)--(MU);
\draw [->, double] (C3)--(C3g);
\draw [->, double] (C3g)--(Ros);
\draw [->, double] (EC3)--(C3);
\draw [->, double] (MU)--(DCUCB);
\draw [->, double] (DCUCB)--(ConH);
\draw [->, double] (2GPCG)--(ConG);
\draw [->, double] (2GPCG)--(1U2U);

% \draw (3,1) node[right]{Cor.~\ref{}};

\end{tikzpicture}
\caption{Implications between the derivability conditions and the versions of G2 for $n, k \geq 1$ and $\Sigma_1$ provability predicates}\label{Fig1}
\end{figure}
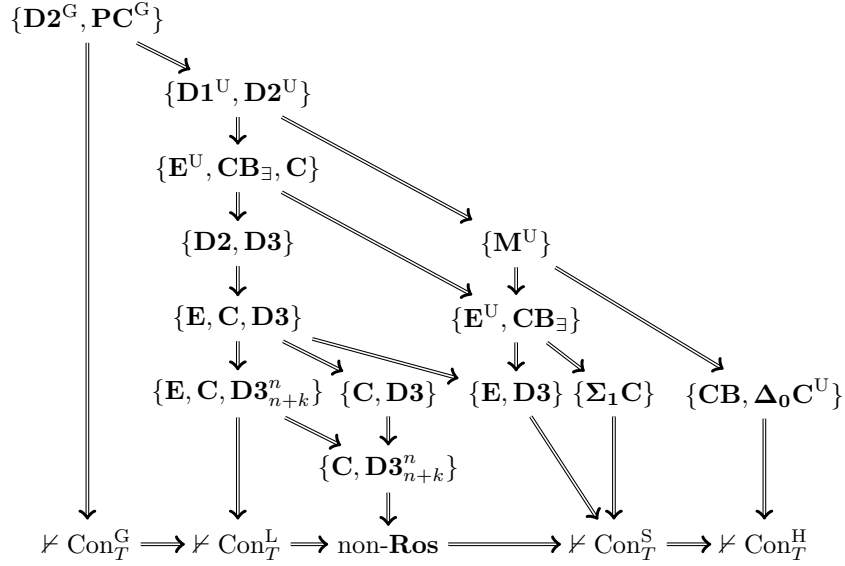

Recently, Haruka Kogure extensively investigated modal logical aspects of $\E$, $\C$, $\ConL_T$, $\Ros$, and $\ConS_T$ from the viewpoint of provability logic. 
Her results will be announced in the near future.

\section{Local derivability conditions}

Let $\LA$ be the language $\{0, \s, +, \times, \leq\}$ of first-order arithmetic. 
For each natural number $n$, let $\num{n}$ denote the numeral $\underbrace{\s(\s(\cdots \s}_{n\ \text{times}}(0) \cdots ))$ of $n$. 
Throughout the present paper, let $S$ and $T$ be any consistent c.e.~extensions of Peano Arithmetic $\PA$ in the language $\LA$ such that $T$ is stronger than $S$. 
The theory $S$ is intended as a meta-theory that reasons about $T$. 
In most cases, $S$ is taken to be either $\PA$ or $T$. 
We say that a formula $\PR_T(x)$ is a \textit{provability predicate} of $T$ if for any $\LA$-formula $\varphi$, we have $T \vdash \varphi$ if and only if $S \vdash \PR_T(\gn{\varphi})$. 
Throughout this paper, we assume that $\PR_T(x)$ always denotes a $\Sigma_1$ provability predicate of $T$. 

In \cite{Kura20}, three types of derivability conditions---local, uniform, and global---were systematically classified and studied.
In this paper, we discuss the local versions in this section, and turn to the uniform versions in the next section. 
We begin by recalling the local derivability conditions. 

\begin{defn}[Local derivability conditions]
\leavevmode
\begin{description} 
	\item [$\D{1}$] If $T \vdash \varphi$, then $S \vdash \PR_T(\gn{\varphi})$ for any formula $\varphi$. 
	\item [$\D{2}$] $S \vdash \PR_T(\gn{\varphi \to \psi}) \to \bigl(\PR_T(\gn{\varphi}) \to \PR_T(\gn{\psi}) \bigr)$ for any formulas $\varphi$ and $\psi$. 
	\item [$\D{3}$] $S \vdash \PR_T(\gn{\varphi}) \to \PR_T(\gn{\PR_T(\gn{\varphi})})$ for any formula $\varphi$. 
	\item [$\SC$] $S \vdash \varphi \to \PR_T(\gn{\varphi})$ for any $\Sigma_1$ sentence $\varphi$. 
	\item [$\M$] If $T \vdash \varphi \to \psi$, then $S \vdash \PR_T(\gn{\varphi}) \to \PR_T(\gn{\psi})$ for any formulas $\varphi$ and $\psi$.  
	\item [$\PC$] $S \vdash \PR_\emptyset(\gn{\varphi}) \to \PR_T(\gn{\varphi})$ for any formula $\varphi$. 
    \item [$\SC^-$] There exists a finite subtheory $T_0$ of $T$ such that for any $\Sigma_1$ sentence $\varphi$, we have $S \vdash \varphi \to \PR_T(\gn{\bigwedge T_0 \to \varphi})$. 
    By convention, we regard $\bigwedge \emptyset$ as $0=0$. 
\end{description}
\end{defn}

Here, $\PR_\emptyset(x)$ denotes a standard provability predicate of pure first-order logic in $\LA$. 
We note that every provability predicate $\PR_T(x)$ satisfies $\D{1}$. 
The implications ``$\D{2} \Rightarrow \M$'', ``$\SC \Rightarrow \D{3}$'', and ``$\SC \Rightarrow \SC^-$'' are obvious. 
Here, ``$\D{2} \Rightarrow \M$'' is an abbreviation for the statement ``every $\Sigma_1$ provability predicate satisfying $\D{2}$ also satisfies $\M$''. 
For the sake of simplicity, we adopt this kind of notation throughout the present paper.

\begin{fact}\label{PC_SC}\leavevmode
\begin{enumerate}
    \item \textup{(\cite[Proposition 2.4.6 and Remark 2.11]{Kura20})} $\PC \Rightarrow \SC^-$. 
    \item \textup{(\cite[Proposition 2.4.6 and Remark 2.11]{Kura20})} $\M$ \& $\SC^- \Rightarrow \SC$. 
    \item \textup{(\cite[Proposition 2.4.6]{Kura20})} $\M$ \& $\SC \Rightarrow \PC$. 
\end{enumerate}
\end{fact}

Therefore, among provability predicates satisfying $\M$, the conditions $\PC$, $\SC$, and $\SC^-$ are equivalent.

Let $\Fml(x)$ be a natural formula saying that $x$ is a G\"odel number of an $\LA$-formula. 
Let $\dot{\neg} \, x$ be a term corresponding to the primitive recursive function calculating the G\"odel number of $\neg \, \varphi$ from that of $\varphi$. 
We may freely use such terms in defining formulas. 
In \cite{Kura20}, the following two consistency statements were distinguished. 

\begin{defn}\leavevmode
\begin{enumerate}
	\item $\ConH_T : \equiv \forall x \, (\Fml(x) \land \PR_T(x) \to \neg \, \PR_T(\dot{\neg} \, x))$. 
	\item $\ConL_T : \equiv \neg \, \PR_T(\gn{0 = 1})$. 
\end{enumerate}
\end{defn}

Here, H and L stand for ``Hilbert and Bernays'' and ``L\"ob'', respectively. 
Several variations of G2 corresponding to these formulations of consistency have been established.

\begin{fact}[Several versions of G2]\label{previous_G2}\leavevmode
\begin{enumerate}
    \item \textup{(L\"ob \cite{Lob55})} $\D{2}$ \& $\D{3} \Rightarrow T \nvdash \ConL_T$. 

    \item \textup{(\cite[Theorems 2.8]{Kura20})} $\M$ \& $\D{3} \Rightarrow T \nvdash \ConH_T$. 

    \item \textup{(\cite[Theorems 2.10 and Remark 2.11]{Kura20})} $\SC^- \Rightarrow T \nvdash \ConH_T$. 
\end{enumerate}
\end{fact}

The third clause is a refinement of G2 of Jeroslow \cite{Jer73} and Kreisel and Takeuti \cite{KT74} stating ``$\SC \Rightarrow T \nvdash \ConH_T$.''

\subsection{Refinements}

As stated in the introduction, in this paper we introduce not only the derivability conditions and consistency statements as above, but also several additional principles and rules corresponding to axioms and rules that have been studied in the context of non-normal modal logic, and we examine their effects.
The rule \textsc{RE}: $\dfrac{A \leftrightarrow B}{\Box A \leftrightarrow \Box B}$ and the axiom \textsc{C}: $\Box A \land \Box B \to \Box (A \land B)$ have been extensively studied in non-normal modal logic weaker than $\mathsf{K}$.
Moreover, the axiom \textsc{P}: $\neg \, \Box \bot$ corresponds to the consistency statement $\ConL_T$, while the axiom \textsc{D}: $\Box A \to \neg \, \Box \neg\, A$ corresponds to the following schematic consistency statement $\ConS_T$, introduced in \cite{Kura21}.
Additionally, we introduce a rule $\Ros$, which corresponds to the modal rule \textsc{Ros}: $\dfrac{\neg\, A}{\neg\, \Box A}$ introduced and investigated in \cite{Kura24}.

\begin{defn}\leavevmode
\begin{description} 
	\item [$\E$] If $T \vdash \varphi \leftrightarrow \psi$, then $S \vdash \PR_T(\gn{\varphi}) \leftrightarrow \PR_T(\gn{\psi})$ for any formulas $\varphi$ and $\psi$. 

    \item [$\C$] $S \vdash \PR_T(\gn{\varphi}) \land \PR_T(\gn{\psi}) \to \PR_T(\gn{\varphi \land \psi})$ for any formulas $\varphi$ and $\psi$. 

 	\item $\ConS_T : \equiv \{\PR_T(\gn{\varphi}) \to \neg \, \PR_T(\gn{\neg\, \varphi}) \mid \varphi$ is a formula$\}$. 
    
	\item [$\Ros$] If $T \vdash \neg \, \varphi$, then $T \vdash \neg \, \PR_T(\gn{\varphi})$ for any formula $\varphi$. 
\end{description}
\end{defn}

The implication ``$\M \Rightarrow \E$'' and the equivalence ``$\D{2} \Leftrightarrow \M$ \& $\C$'' hold obviously.

We show the following relations between the four consistency statements $\ConH_T$, $\ConS_T$, $\Ros$, and $\ConL_T$. 
In particular, it is noteworthy that $\C$ and $\E$ contribute to enhancing the strength of the corresponding consistency statements.

\begin{prop}\label{EC}\leavevmode
\begin{enumerate}
    \item $\PA + \ConH_T \vdash \ConS_T$.
    \item $T \vdash \ConS_T \Rightarrow \Ros$. 
    \item $\Ros$ \& $\C \Rightarrow T \vdash \ConS_T$. 
    \item $\Ros \Rightarrow T \vdash \ConL_T$. 
    \item $T \vdash \ConL_T$ \& $\E \Rightarrow \Ros$. 
\end{enumerate}
\end{prop}
\begin{proof}
1. Obvious. 

\medskip

2. Suppose $T \vdash \neg\, \varphi$. 
By applying $\D{1}$, we have $S \vdash \PR_T(\gn{\neg\, \varphi})$. 
Since $T \vdash \ConS_T$, we obtain $T \vdash \neg\, \PR_T(\gn{\varphi})$. 

\medskip

3. From $\C$, we have $S \vdash \PR_T(\gn{\varphi}) \land \PR_T(\gn{\neg\, \varphi}) \to \PR_T(\gn{\varphi \land \neg\, \varphi})$. 
Since $T \vdash \neg\, (\varphi \land \neg \, \varphi)$, we have $T \vdash \neg \, \PR_T(\gn{\varphi \land \neg\, \varphi})$. 
It follows that $T \vdash \PR_T(\gn{\varphi}) \to \neg\, \PR_T(\gn{\neg\, \varphi})$. 

\medskip

4. Immediate from $T \vdash \neg \, 0=1$. 

\medskip

5. Suppose $T \vdash \neg\, \varphi$. 
Since $T \vdash \varphi \leftrightarrow 0 = 1$, by applying $\E$, we have $S \vdash \PR_T(\gn{\varphi}) \leftrightarrow \PR_T(\gn{0=1})$. 
Since $T \vdash \neg \, \PR_T(\gn{0=1})$, we conclude $T \vdash \neg \, \PR_T(\gn{\varphi})$. 
\end{proof}

In what follows, we refine several results previously obtained. 
Although the proofs are essentially the same as those presented in \cite{Kura20}, we provide them here again for clarity and completeness.
The following proposition is a refinement of Fact \ref{PC_SC}.2. 

\begin{prop}
    $\E$ \& $\SC^- \Rightarrow \SC$. 
\end{prop}
\begin{proof}
Let $\sigma$ be any $\Sigma_1$ sentence. 
By applying $\SC^-$, we find a finite subtheory $T_0$ of $T$ such that $S \vdash \sigma \to \PR_T(\gn{\bigwedge T_0 \to \sigma})$. 
Since $T \vdash \left(\bigwedge T_0 \to \sigma \right) \leftrightarrow \sigma$, by applying $\E$, we obtain $S \vdash \PR_T(\gn{\bigwedge T_0 \to \sigma}) \leftrightarrow \PR_T(\gn{\sigma})$. 
We conclude $S \vdash \sigma \to \PR_T(\gn{\sigma})$. 
\end{proof}

We refine Fact \ref{previous_G2}.

\begin{thm}\label{G2_refine_E}\leavevmode
\begin{enumerate}
    \item  $\E$ \& $\D{3} \Rightarrow T \nvdash \ConS_T$. 
    \item  $\SC^- \Rightarrow T \nvdash \ConS_T$. 
\end{enumerate}
\end{thm}
\begin{proof}
1. Suppose $\PR_T(x)$ satisfies $\E$ and $\D{3}$. 
By the Fixed Point Theorem (cf.~\cite{HP93,Lin03}), we find a sentence $\varphi$ such that $\PA \vdash \varphi \leftrightarrow \neg\, \PR_T(\gn{\varphi})$. 
By applying $\E$, we have $S \vdash \PR_T(\gn{\varphi}) \leftrightarrow \PR_T(\gn{\neg\, \PR_T(\gn{\varphi})})$. 
From $\D{3}$, we have $S \vdash \PR_T(\gn{\varphi}) \to \PR_T(\gn{\PR_T(\gn{\varphi})})$. 
Suppose, towards a contradiction, that $T \vdash \ConS_T$. 
Then, $T \vdash \neg\, \PR_T(\gn{\varphi})$. 
By the fixed point equivalence, $T \vdash \varphi$. 
By $\D{1}$, we have $S \vdash \PR_T(\gn{\varphi})$. 
This contradicts the consistency of $T$. 

\medskip

2. Let $T_0$ be a finite subtheory of $T$ witnessing $\SC^-$. 
By the Fixed Point Theorem, we find a $\Sigma_1$ sentence $\sigma$ such that $\PA \vdash \sigma \leftrightarrow \PR_T\bigl(\gn{\neg\, \bigl(\bigwedge T_0 \to \sigma \bigr)}\bigr)$. 
It follows from $\SC^-$ that $S \vdash \sigma \to \PR_T\bigl(\gn{\bigwedge T_0 \to \sigma}\bigr)$. 
Suppose, towards a contradiction, that $T \vdash \ConS_T$. 
Then, $T \vdash \neg\, \sigma$ and also $T \vdash \neg\, \bigl(\bigwedge T_0 \to \sigma \bigr)$. 
By $\D{1}$, we have $S \vdash \PR_T\bigl(\gn{\neg\, \bigl(\bigwedge T_0 \to \sigma \bigr)}\bigr)$. 
By the fixed point equivalence, $S \vdash \sigma$. 
This contradicts the consistency of $T$. 
\end{proof}

The following corollary follows from the combination of Proposition \ref{EC} and Theorem \ref{G2_refine_E}. 

\begin{cor}\label{G2_refine_EC}\leavevmode
\begin{enumerate}
    \item  $\E$ \& $\C$ \& $\D{3} \Rightarrow T \nvdash \ConL_T$. 
    \item  $\C$ \& $\SC^- \Rightarrow$ \textup{non}-$\Ros$. 
\end{enumerate}
\end{cor}

In particular, Corollary \ref{G2_refine_EC}.1 is a refinement of the well-known formulation of G2 presented in Fact \ref{previous_G2}.1, and is of particular interest.
Furthermore, we will prove that $\{\E, \C, \D{3}\}$ is strictly weaker than $\{\D{2}, \D{3}\}$ (Theorem \ref{ex1}). 

From Corollary \ref{G2_refine_EC}.2, we have that ``$\C$ \& $\SC \Rightarrow \PR_T(x)$ does not satisfy $\Ros$''. 
Interestingly, this implication is refined by proving the following new version of G2 concerning the newly introduced principles $\C$, $\D{3}^n_m$, and consistency rule $\Ros$. 
For each provability predicate $\PR_T(x)$, we recursively define $\PR_T^n(\gn{\varphi})$ as follows: $\PR_T^0(\gn{\varphi})$ is $\varphi$; $\PR_T^{n+1}(\gn{\varphi})$ is $\PR_T(\gn{\PR_T^n(\gn{\varphi})})$. 

\begin{defn}
Let $m, n$ be natural numbers. 
\begin{description} 
	\item [$\D{3}^n_m$] $S \vdash \PR_T^n(\gn{\varphi}) \to \PR_T^m(\gn{\varphi})$ for any formula $\varphi$. 
\end{description}
\end{defn}

The usual $\D{3}$ is $\D{3}^1_2$. 
The condition $\D{3}^n_m$ corresponds to the modal axiom $\Box^n A \to \Box^m A$ which is studied in \cite{Kogu,KS} in the context of non-normal modal logic.

\begin{thm}\label{G2_C3}
Let $n, k \geq 1$. 
Then, $\C$ \& $\D{3}^n_{n+k} \Rightarrow$ \textup{non}-$\Ros$. 
\end{thm}
\begin{proof}
Suppose, towards a contradiction, that $\PR_T(x)$ satisfies $\C$, $\D{3}^n_{n+k}$, and $\Ros$. 
By the Fixed Point Theorem, we find a sentence $\varphi$ satisfying
\begin{equation}\label{eq15}
    \PA \vdash \varphi \leftrightarrow \PR_T (\gn{\neg \, \xi}),  
\end{equation}
where $\xi$ is the sentence
\[
    \bigwedge_{0 \leq i \leq n+k-1} \PR_T^i(\gn{\varphi}).
\]
From $\D{3}^n_{n+k}$ and (\ref{eq15}), we get
\begin{align*}
    S \vdash \xi & \to \varphi \land \bigwedge_{1 \leq i \leq n+k-1} \PR_T^i(\gn{\varphi})\\
    & \to \PR_T (\gn{\neg \, \xi}) \land \bigwedge_{1 \leq i \leq n+k} \PR_T^i(\gn{\varphi}). 
\end{align*}
By the $n+k$ times applications of $\C$, we have
\[
    S \vdash \xi \to \PR_T\left(\gn{\neg\, \xi \land \xi} \right). 
\]
Since $T \vdash \neg\, (\neg \, 
\xi \land \xi)$, by applying $\Ros$, we obtain $T \vdash \neg\, \PR_T(\gn{\neg \, \xi \land \xi})$. 
Thus, $T \vdash \neg\, \xi$, and hence $S \vdash \PR_T(\gn{\neg\, \xi})$. 
By (\ref{eq15}), $S \vdash \varphi$. 
By applying $\D{1}$, we have $\displaystyle S \vdash \bigwedge_{0 \leq i \leq n+k-1} \PR_T^i(\gn{\varphi})$. 
So, $T \vdash \xi$. 
This contradicts the consistency of $T$. 
\end{proof}

The following corollary is a refinement of Corollary \ref{G2_refine_EC}, which follows from the combination of Proposition \ref{EC}.5 and Theorem \ref{G2_C3}.

\begin{cor}\label{G2_EC3}
Let $n, k \geq 1$.
Then, $\E$ \& $\C$ \& $\D{3}^n_{n+k} \Rightarrow T \nvdash \ConL_T$. 
\end{cor}

\subsection{L\"ob's theorem}

In usual, L\"ob's theorem holds for provability predicates satisfying $\D{2}$ and $\D{3}$. 
Then G2 for $\ConL_T$ is a direct consequence of L\"ob's theorem. 
In Corollary \ref{G2_EC3}, we showed that a weaker $\{\E, \C, \D{3}^n_{n+k}\}$ with $n, k \geq 1$ is sufficient for G2 for $\ConL_T$. 
We show that the set $\{\E, \C, \D{3}^n_{n+k}\}$ is also sufficient for the following weak variant of L\"ob's theorem. 

\begin{thm}\label{Lob1}
Let $n, k \geq 1$. 
Suppose that $\PR_T(x)$ satisfies $\E$, $\C$, and $\D{3}^n_{n+k}$. 
Let $\varphi$ be any formula. 
If $T \vdash \PR_T(\gn{\varphi \land \bigwedge_{1 \leq i \leq n+k-1}\PR_T^i(\gn{\chi})}) \to \varphi$ holds for all sentences $\chi$, then $T \vdash \varphi$. 
\end{thm}
\begin{proof}
Suppose that $\PR_T(x)$ satisfies $\E$, $\C$, and $\D{3}^n_{n+k}$. 
Assume that $T \vdash \PR_T(\gn{\varphi \land \bigwedge_{1 \leq i \leq n+k-1}\PR_T^i(\gn{\chi})}) \to \varphi$ holds for all sentences $\chi$. 
By the Fixed Point Theorem, we find a sentence $\psi$ satisfying the following equivalence: 
\[
    \PA \vdash \psi \leftrightarrow (\xi \to \varphi),  
\]
where $\xi$ is the sentence
\[
    \displaystyle \bigwedge_{1 \leq i \leq n+k-1} \PR_T^i(\gn{\psi}).
\]
Since $T \vdash (\psi \land \xi) \leftrightarrow (\varphi \land \xi)$, by applying $\E$, we have
\begin{equation}\label{eq20}
    S \vdash \PR_T(\gn{\psi \land \xi}) \leftrightarrow \PR_T(\gn{\varphi \land \xi}). 
\end{equation}

From $\D{3}^n_{n+k}$, 
\[
    S \vdash \xi \to \bigwedge_{1 \leq i \leq n+k} \PR_T^i(\gn{\psi}). 
\]
By the $n + k - 1$ times applications of $\C$, we get 
\[
    S \vdash \xi \to \PR_T\left(\gn{\bigwedge_{0 \leq i \leq n+k-1} \PR_T^i(\gn{\psi})} \right), 
\]
and equivalently, $S \vdash \xi \to \PR_T(\gn{\psi \land \xi})$. 
It follows from (\ref{eq20}) that $S \vdash \xi \to \PR_T(\gn{\varphi \land \xi})$. 
By the assumption, $T \vdash \PR_T(\gn{\varphi \land \xi}) \to \varphi$, and thus $T \vdash \xi \to \varphi$. 
By the fixed point equivalence, $T \vdash \psi$. 
So, $S \vdash \bigwedge_{1 \leq i \leq n+k-1} \PR_T^i(\gn{\psi})$, that is, $S \vdash \xi$. 
Therefore, we conclude $T \vdash \varphi$. 
\end{proof}

Since $\PA \vdash 0=1 \leftrightarrow (0=1 \land \bigwedge_{1 \leq i \leq n+k-1}\PR_T^i(\gn{\chi}))$ for all $\chi$, Corollary \ref{G2_EC3} is in fact a consequence of Theorem \ref{Lob1}. 
The following corollary is a refinement of the usual L\"ob's theorem. 

\begin{cor}\label{Lob2}
Let $n, k \geq 1$ and let $\varphi$ be any formula. 
Suppose that $\PR_T(x)$ satisfies $\D{2}$ and $\D{3}^n_{n+k}$. 
If $T \vdash \PR_T(\gn{\varphi}) \to \varphi$, then $T \vdash \varphi$. 
\end{cor}
\begin{proof}
Suppose that $\PR_T(x)$ satisfies $\D{2}$ and $\D{3}^n_{n+k}$. 
Assume $T \vdash \PR_T(\gn{\varphi}) \to \varphi$. 
Let $\chi$ be any sentence. 
Since $T \vdash \varphi \land \bigwedge_{1 \leq i \leq n+k-1} \PR_T^i(\gn{\chi}) \to \varphi$, we have $S \vdash \PR_T(\gn{\varphi \land \bigwedge_{1 \leq i \leq n+k-1} \PR_T^i(\gn{\chi})}) \to \PR_T(\gn{\varphi})$ from $\D{1}$ and $\D{2}$. 
Hence, $T \vdash \PR_T(\gn{\varphi \land \bigwedge_{1 \leq i \leq n+k-1} \PR_T^i(\gn{\chi})}) \to \varphi$. 
Since $\PR_T(x)$ satisfies $\E$ and $\C$, it follows from Theorem \ref{Lob1} that $T \vdash \varphi$. 
\end{proof}

The theorem shows that L\"ob's theorem still holds even when a weakened variant of $\D{3}$ is used. 
However, this does not mean that the formalized version of L\"ob's theorem $S \vdash \PR_T(\gn{\PR_T(\gn{\varphi}) \to \varphi}) \to \PR_T(\gn{\varphi})$ generally holds when $\D{3}$ is weakened.
Kurahashi \cite{Kura18} proved that for each $n \geq 1$, for a provability predicate satisfying $\D{2}$, the condition $\D{3}^1_{n+1}$ is equivalent to the provability of a weak form of the formalized L\"ob’s theorem, namely $\PR_T(\gn{\PR_T^n(\gn{\varphi}) \to \varphi}) \to \PR_T(\gn{\varphi})$. 
Moreover, it was proved that there exists a $\Sigma_2$ provability predicate whose provability logic is exactly Sacchetti's logic $\mathsf{K}+ \Box(\Box^n A \to A) \to \Box A$ \cite{Sacc01}. 

\begin{prob}
For $n, k \geq 1$, can one specify a version of the formalized L\"ob’s theorem corresponding to $\D{3}^n_{n+k}$, and determine the associated provability logic?
Moreover, can the provability predicate corresponding to that provability logic be found as a $\Sigma_1$ formula?
\end{prob}

Relating to this problem, Kogure \cite{Kogu} proved that for each $n, k \geq 1$, there exists a $\Sigma_1$ provability predicate whose provability logic is exactly $\mathsf{N} + (\Box^n A\to \Box^k A)$.

\section{Uniform derivability conditions}

In this subsection, we discuss uniform derivability conditions. 
This section consists of two main results. 
The first one is to show that the condition $\M$ in Hilbert and Bernays’ version of G2 is redundant. 
The second one is to refine the implication $\MU \Rightarrow \SCU$ of Fact \ref{uniform}.4. 
We begin by listing the uniform derivability conditions introduced in \cite{Kura20}.

\begin{defn}[Uniform derivability conditions]\leavevmode
\begin{description}
	\item [$\DU{1}$] If $T \vdash \forall \vec{x}\,\varphi(\vec{x})$, then $S \vdash \forall \vec{x}\, \PR_T(\gn{\varphi(\vec{\dot{x}})})$ for any formula $\varphi(\vec{x})$. 

	\item [$\DU{2}$] $S \vdash \forall \vec{x}\, \bigl(\PR_T(\gn{\varphi(\vec{\dot{x}}) \to \psi(\vec{\dot{x}})}) \to (\PR_T(\gn{\varphi(\vec{\dot{x}})}) \to \PR_T(\gn{\psi(\vec{\dot{x}})}))\bigr)$ for any formulas $\varphi(\vec{x})$ and $\psi(\vec{x})$. 

	\item [$\DU{3}$] $S \vdash \forall \vec{x}\, (\PR_T(\gn{\varphi(\vec{\dot{x}})}) \to \PR_T(\gn{ \PR_T(\gn{\varphi(\vec{\dot{x}})})}))$ for any formula $\varphi(\vec{x})$. 

	\item [$\GCU$] $S \vdash \forall \vec{x}\, (\varphi(\vec{x}) \to \PR_T(\gn{\varphi(\vec{\dot{x}})}))$ for any $\Gamma$ formula $\varphi(\vec{x})$. 

	\item [$\MU$] If $T \vdash \forall \vec{x} (\varphi(\vec{x}) \to \psi(\vec{x}))$, then $S \vdash \forall \vec{x} \bigl(\PR_T(\gn{\varphi(\vec{\dot{x}})}) \to \PR_T(\gn{\psi(\vec{\dot{x}})}) \bigl)$\\
for any formulas $\varphi(\vec{x})$ and $\psi(\vec{x})$. 

	\item [$\CB$] $S \vdash \PR_T(\gn{\forall \vec{x}\, \varphi(\vec{x})}) \to \forall \vec{x}\, \PR_T(\gn{\varphi(\vec{\dot{x}})})$ for any formula $\varphi(\vec{x})$. 

	\item [$\PCU$] $S \vdash \forall \vec{x} \bigl(\PR_\emptyset(\gn{\varphi(\vec{\dot{x}})}) \to \PR_T(\gn{\varphi(\vec{\dot{x}})}) \bigr)$ for any formula $\varphi(\vec{x})$. 
\end{description}
\end{defn}

Here $\gn{\varphi(\dot{x})}$ is a term corresponding to a primitive recursive function calculating the G\"odel number of the formula $\varphi(\overline{n})$ from $n$. 
CB stands for ``Converse Barcan''. 
The implications $\SCU \Rightarrow \DU{3}$ and $\DU{1}$ \& $\DU{2} \Rightarrow \MU$ are obvious. 
The following fact presents more implications.

\begin{fact}\label{uniform}\leavevmode
\begin{enumerate}
    \item \textup{(\cite[Proposition 2.14.1]{Kura20})} $\CB \Rightarrow \DU{1}$. 
    \item \textup{(\cite[Proposition 2.14.2]{Kura20})} $\MU \Rightarrow \CB$. 
    \item \textup{(\cite[Proposition 2.14.3]{Kura20})} $\DU{2}$ \& $\PCU \Rightarrow \DU{1}$. 
    \item \textup{(\cite[Theorem 2.20]{Kura20})} $\MU \Rightarrow \SCU$. 
    \item \textup{(\cite[Corollary 2.23]{Kura20})} $\MU \Rightarrow \PCU$. 
\end{enumerate}
\end{fact}

Fact \ref{uniform}.4 is a refinement of Buchholz’s observation ``$\DU{1}$ \& $\DU{2} \Rightarrow \SCU$'' \cite{Buc93} (see also Rautenberg’s textbook \cite{Rau10}).

We introduce two new principles: $\EU$ and $\CBE$. 
Furthermore, we introduce stronger variants $\CBP$ and $\CBEP$ of $\CB$ and $\CBE$, respectively. 

\begin{defn}\leavevmode
\begin{description}
	\item [$\EU$] If $T \vdash \forall \vec{x} (\varphi(\vec{x}) \leftrightarrow \psi(\vec{x}))$, then $S \vdash \forall \vec{x} \bigl(\PR_T(\gn{\varphi(\vec{\dot{x}})}) \leftrightarrow \PR_T(\gn{\psi(\vec{\dot{x}})}) \bigl)$\\
for any formulas $\varphi(\vec{x})$ and $\psi(\vec{x})$. 

	\item [$\CBE$] $S \vdash \exists \vec{x}\, \PR_T(\gn{\varphi(\vec{\dot{x}})}) \to \PR_T(\gn{\exists \vec{x}\, \varphi(\vec{x})})$ for any formula $\varphi(\vec{x})$. 

	\item [$\CBP$] $S \vdash \PR_T(\gn{\forall \vec{x}\, \varphi(\vec{x}, \vec{\dot{y}})}) \to \forall \vec{x}\, \PR_T(\gn{\varphi(\vec{\dot{x}}, \vec{\dot{y}})})$ for any formula $\varphi(\vec{x}, \vec{y})$. 

	\item [$\CBEP$] $S \vdash \exists \vec{x}\, \PR_T(\gn{\varphi(\vec{\dot{x}}, \vec{\dot{y}})}) \to \PR_T(\gn{\exists \vec{x}\, \varphi(\vec{x}, \vec{\dot{y}})})$ for any formula $\varphi(\vec{x}, \vec{y})$. 
\end{description}
\end{defn}

\begin{prop}\label{EU_to_D1U}
\begin{enumerate}
    \item $\EU \Rightarrow \DU{1}$. 
    \item $\EU$ \& $\PCU \Rightarrow \SCU$. 
    \item $\CBE$ \& $\DCU \Rightarrow \SC$. 
    \item $\CBEP$ \& $\DCU \Rightarrow \SC^\mathrm{U}$. 
\end{enumerate}
\end{prop}
\begin{proof}
1. Suppose $\PR_T(x)$ satisfies $\EU$. 
    Assume $T \vdash \varphi(\vec{x})$, then $T \vdash 0=0 \leftrightarrow \varphi(\vec{x})$. 
    By $\EU$, we have $S \vdash \PR_T(\gn{0=0}) \leftrightarrow \PR_T(\gn{\varphi(\vec{\dot{x}})})$.
    Since $S \vdash \PR_T(\gn{0=0})$, we conclude $S \vdash \PR_T(\gn{\varphi(\vec{\dot{x}})})$.

  \medskip

2. Suppose $\PR_T(x)$ satisfies $\EU$ and $\PCU$. 
Let $\varphi(\vec{x})$ be any $\Sigma_1$ formula. 
We have $S \vdash \varphi(\vec{x}) \to \PR_{\Q}(\gn{\varphi(\dot{x})})$ and hence $S \vdash \varphi(\vec{x}) \to \PR_\emptyset(\gn{\bigwedge \Q \to \varphi(\vec{\dot{x}})})$ by the formalized deduction theorem. 
From $\PC$, we get $S \vdash \varphi(\vec{x}) \to \PR_T(\gn{\bigwedge \Q \to \varphi(\vec{\dot{x}})})$. 
Since $T \vdash \bigl(\bigwedge \Q \to \varphi(\vec{x}) \bigr) \leftrightarrow \varphi(\vec{x})$, by applying $\EU$, we obtain $S \vdash \PR_T(\gn{\bigwedge \Q \to \varphi(\vec{\dot{x}})}) \leftrightarrow \PR_T(\gn{\varphi(\vec{\dot{x}})})$. 
We conclude $S \vdash \varphi(\vec{x}) \to \PR_T(\gn{\varphi(\vec{\dot{x}})})$. 

\medskip

3 and 4 are clear, but it is important to distinguish the effects of $\CBE$ and $\CBEP$.
\end{proof}

From the perspective of non-normal modal logic, $\CB$ and $\CBE$ correspond to the axioms $\Box (A \land B) \to \Box A \land \Box B$ and $\Box A \lor \Box B \to \Box (A \lor B)$, respectively. 
Moreover, each one of these axioms is equivalent to \textsc{RM} over the rule \textsc{RE}. 
This observation can in fact be formulated as the following proposition.

\begin{prop}\label{MU_eq}
The following conditions are pairwise equivalent: 
\begin{enumerate}
    \item $\MU$. 
    \item $\EU$ \& \ $\CBP$. 
    \item $\EU$ \& \ $\CBEP$. 
\end{enumerate}
\end{prop}
\begin{proof}
Note that the implication $\MU \Rightarrow \EU$ obviously holds. 

\medskip

$(1 \Rightarrow 2)$: This is a slight refinement of Fact \ref{uniform}.2. 
Suppose $\PR_T(x)$ satisfies $\MU$. 
Let $\varphi(\vec{x}, \vec{y})$ be any formula. 
    Since $T \vdash \forall \vec{x}\, \varphi(\vec{x}, \vec{y}) \to \varphi(\vec{x}, \vec{y})$, by applying $\MU$, we get $S \vdash \PR_T(\gn{\forall \vec{x}\, \varphi(\vec{x}, \vec{\dot{y}})}) \to \PR_T(\gn{\varphi(\vec{\dot{x}}, \vec{\dot{y}})})$. 
Thus, $S \vdash \PR_T(\gn{\forall \vec{x}\, \varphi(\vec{x}, \vec{\dot{y}})}) \to \forall \vec{x}\, \PR_T(\gn{\varphi(\vec{\dot{x}}, \vec{\dot{y}})})$. 

\medskip

$(2 \Rightarrow 1)$: Suppose $\PR_T(x)$ satisfies $\EU$ and $\CBP$. 
Let $\varphi(\vec{x})$ and $\psi(\vec{x})$ be any formulas such that $T \vdash \varphi(\vec{x}) \to \psi(\vec{x})$. 
Let $\xi(\vec{x}, y)$ denote the formula
\[
    \bigl(y = 0 \to \psi(\vec{x}) \bigr) \land \bigl(y \neq 0 \to (\psi(\vec{x}) \to \varphi(\vec{x})) \bigr).
\]
Then, we have $T \vdash \xi(\vec{x}, 0) \leftrightarrow \psi(\vec{x})$. 
Also, 
\begin{align*}
    T \vdash \forall y\, \xi(\vec{x}, y) & \leftrightarrow \psi(\vec{x}) \land (\psi(\vec{x}) \to \varphi(\vec{x})), \\
    & \leftrightarrow \varphi(\vec{x}). 
\end{align*}
By applying $\EU$, we have $S \vdash \PR_T(\gn{\xi(\vec{\dot{x}}, 0)}) \leftrightarrow \PR_T(\gn{\psi(\vec{\dot{x}})})$ and $S \vdash \PR_T(\gn{\forall y\, \xi(\vec{\dot{x}}, y)}) \leftrightarrow \PR_T(\gn{\varphi(\vec{\dot{x}})})$. 
From $\CBP$, we get $S \vdash \PR_T(\gn{\forall y\, \xi(\vec{\dot{x}}, y)}) \to  \forall y\, \PR_T(\gn{\xi(\vec{\dot{x}}, \dot{y})})$, and hence $S \vdash \PR_T(\gn{\forall y\, \xi(\vec{\dot{x}}, y)}) \to \PR_T(\gn{\xi(\vec{\dot{x}}, 0)})$. 
It follows that $S \vdash \PR_T(\gn{\varphi(\vec{\dot{x}})}) \to \PR_T(\gn{\psi(\vec{\dot{x}})})$.

\medskip

$(1 \Rightarrow 3)$: 
This is proved as in the proof of $(1 \Rightarrow 2)$. 

\medskip

$(3 \Rightarrow 1)$: 
Suppose $\PR_T(x)$ satisfies $\EU$ and $\CBEP$ and let $\varphi(\vec{x})$ and $\psi(\vec{x})$ be formulas such that $T \vdash \varphi(\vec{x}) \to \psi(\vec{x})$. 
Let $\eta(\vec{x}, y)$ denote the formula
\[
    \bigl(y = 0 \land \varphi(\vec{x}) \bigr) \lor \bigl(y \neq 0 \land (\neg \, \varphi(\vec{x}) \land \psi(\vec{x})) \bigr).
\]
Then, we have $T \vdash \eta(\vec{x}, 0) \leftrightarrow \varphi(\vec{x})$. 
Also, by the law of excluded middle, 
\begin{align*}
    T \vdash \exists y\, \eta(\vec{x}, y) & \leftrightarrow \varphi(\vec{x}) \lor (\neg\, \varphi(\vec{x}) \land \psi(\vec{x})), \\
    & \leftrightarrow \psi(\vec{x}). 
\end{align*}
By applying $\EU$, we have $S \vdash \PR_T(\gn{\eta(\vec{\dot{x}}, 0)}) \leftrightarrow \PR_T(\gn{\varphi(\vec{\dot{x}})})$ and $S \vdash \PR_T(\gn{\exists y\, \eta(\vec{\dot{x}}, y)}) \leftrightarrow \PR_T(\gn{\psi(\vec{\dot{x}})})$. 
From $\CBEP$, we get $S \vdash \exists y\, \PR_T(\gn{\eta(\vec{\dot{x}}, \dot{y})}) \to \PR_T(\gn{\exists y\, \eta(\vec{\dot{x}}, y)})$, and hence $S \vdash \PR_T(\gn{\eta(\vec{\dot{x}}, 0)}) \to \PR_T(\gn{\exists y\, \eta(\vec{\dot{x}}, y)})$. 
It follows that $S \vdash \PR_T(\gn{\varphi(\vec{\dot{x}})}) \to \PR_T(\gn{\psi(\vec{\dot{x}})})$. 
\end{proof}

The following proposition is proved in the same way as in the proof of Proposition \ref{MU_eq}. 

\begin{prop}\label{CB_local}
\begin{enumerate}
    \item $\E$ \& \ $\CB \Rightarrow \M$. 
    \item $\E$ \& \ $\CBE \Rightarrow \M$. 
\end{enumerate}
\end{prop}

\subsection{A refinement of Hilbert and Bernays' G2}

Hilbert and Bernays \cite{HB39} proved that if a $\Sigma_1$ provability predicate $\PR_T(x)$ satisfies the following three conditions $\HB{1}$, $\HB{2}$ and $\HB{3}$, then $T \nvdash \ConH_T$. 

\begin{description}
	\item [$\HB{1}$] If $T \vdash \varphi \to \psi$, then $T \vdash \PR_T(\gn{\varphi}) \to \PR_T(\gn{\psi})$. 
	\item [$\HB{2}$] $T \vdash \PR_T(\gn{\neg \, \varphi(x)}) \to \PR_T(\gn{\neg \, \varphi(\dot{x})})$. 
	\item [$\HB{3}$] $T \vdash f(x) = 0 \to \PR_T(\gn{f(\dot{x}) = 0})$ for every primitive recursive term $f(x)$.
\end{description}

In our framework, $\HB{1}$, $\HB{2}$ and $\HB{3}$ correspond to $\M$, $\CB$ and $\DCU$, respectively. 
So, a version of G2 due to Hilbert and Bernays is formulated as follows. 

\begin{fact}[Hilbert and Bernays \cite{HB39}]\label{UHB}
$\M$ \&\ $\CB$ \& $\DCU \Rightarrow T \nvdash \ConH_T$. 
\end{fact}

In \cite[Problem 2.17.1]{Kura20}, the author asked whether there exists a $\Sigma_1$ provability predicate satisfying $\CB$ and $\DCU$ such that $T \vdash \ConH_T$. 
We refine Fact \ref{UHB} by showing that no such provability predicate exists, and in doing so, we analyze the roles of $\CB$ and $\DCU$ through the uniform $\Delta_0$ reflection principle $\RFN_T(\Delta_0)$. 
Here, $\RFN_T(\Delta_0)$ is the set $\{\forall \vec{x}\, \bigl(\PR_T(\gn{\varphi(\vec{\dot{x}})}) \to \varphi(\vec{x}) \bigr) \mid \varphi(\vec{x})$ is a $\Delta_0$ formula$\}$. 

\begin{prop}\label{DCU_to_RFN}
$\DCU \Rightarrow S + \ConH_T \vdash \RFN_T(\Delta_0)$. 
\end{prop}
\begin{proof}
Suppose $\PR_T(x)$ satisfies $\DCU$. 
Let $\varphi(\vec{x})$ be any $\Delta_0$ formula. 
Since $\neg\, \varphi(\vec{x})$ is also a $\Delta_0$ formula, we have $S \vdash \neg\, \varphi(\vec{x}) \to \PR_T(\gn{\neg\, \varphi(\vec{\dot{x}})})$ by $\DCU$. 
Since $S + \ConH_T \vdash \PR_T(\gn{\varphi(\vec{\dot{x}})}) \to \neg\, \PR_T(\gn{\neg\, \varphi(\vec{\dot{x}})})$, we conclude $S + \ConH_T \vdash \PR_T(\gn{\varphi(\vec{\dot{x}})}) \to \varphi(\vec{x})$. 
\end{proof}

The following proposition is itself a version of G2. 

\begin{prop}\label{CB_to_RFN}
$\CB \Rightarrow T \nvdash \RFN_T(\Delta_0)$. 
\end{prop}
\begin{proof}
Suppose $\PR_T(x)$ satisfies $\CB$. 
Let $\mathsf{pnf}(x)$ be a term corresponding to a primitive recursive function calculating the G\"odel number of a formula in prenex normal form $\PA$-equivalent to $\varphi$ from that of $\varphi$. 
By the Fixed Point Theorem, we find a $\Pi_1$ sentence $\varphi$ satisfying $\PA \vdash \varphi \leftrightarrow \neg\, \PR_T(\mathsf{pnf}(\gn{\varphi}))$. 
Let $\delta(\vec{x})$ be a $\Delta_0$ formula such that $\PA \vdash \mathsf{pnf}(\gn{\varphi}) =  \gn{\forall \vec{x}\, \delta(\vec{x})}$. 
Then, we have $\PA \vdash \forall \vec{x}\, \delta(\vec{x}) \leftrightarrow \neg\, \PR_T(\gn{\forall \vec{x}\, \delta(\vec{x})})$. 

From $\CB$, we have $S \vdash \PR_T(\gn{\forall \vec{x}\, \delta(\vec{x})}) \to \PR_T(\gn{\delta(\vec{\dot{x}})})$. 
Suppose, towards a contradiction, that $T \vdash \RFN_T(\Delta_0)$. 
Then, $T \vdash \PR_T(\gn{\forall \vec{x}\, \delta(\vec{x})}) \to \delta(\vec{x})$, and thus $T \vdash \PR_T(\gn{\forall \vec{x}\, \delta(\vec{x})}) \to \forall \vec{x}\, \delta(\vec{x})$ by generalization. 
From the fixed point equivalence, $T \vdash \neg \, \forall \vec{x}\, \delta(\vec{x}) \to \forall \vec{x}\, \delta(\vec{x})$, and thus $T \vdash \forall \vec{x}\, \delta(\vec{x})$. 
By applying $\D{1}$, we have $S \vdash \PR_T(\gn{\forall \vec{x}\, \delta(\vec{x})})$. 
By the fixed point equivalence again, $S \vdash \neg\, \forall \vec{x}\, \delta(\vec{x})$, a contradiction. 
\end{proof}

By combining these propositions, we obtain the following theorem. 

\begin{thm}\label{HB_refine}
$\CB$ \& $\DCU \Rightarrow T \nvdash \ConH_T$. 
\end{thm}

By combining Theorem \ref{G2_refine_E}.2 and Proposition \ref{EU_to_D1U}.3, we have ``$\CBE$ \& $\DCU \Rightarrow T \nvdash \ConS_T$'', which may be understood as a dual version of Theorem \ref{HB_refine}.

\subsection{A refinement of the proof of $\SC$}

Fact \ref{uniform}.4 states the implication $\MU \Rightarrow \SCU$. 
By combining this with Proposition \ref{MU_eq}, we have that $\EU$ \& $\CBEP  \Rightarrow \SCU$. 
In this subsection, by examining the proof of the latter implication, we show that most parts of the proof require only the assumption $\EU$, and in particular, we obtain the following new theorem.

\begin{thm}\label{EU_to_SCU}
  $\EU$ \&  $\CBE \Rightarrow \SC$. 
\end{thm}

Before proving the theorem, we prepare several lemmata.

\begin{lem}\label{Lem1}
Suppose that $\PR_T(x)$ satisfies $\EU$. 
Then, for any families $\{x_{i,j}\}_{\substack{i \leq k \\ j \leq l}}$ and $\{z_{i,j}\}_{\substack{i \leq k \\ j \leq l}}$ of variables, we have
\[
    S \vdash \bigvee_{i \leq k} \bigwedge_{j \leq l} (z_{i,j} = x_{i,j}) \to \PR_T\Bigl(\gn{\bigvee_{i \leq k} \bigwedge_{j \leq l} (\dot{z_{i,j}} = \dot{x_{i,j}})}\Bigr).
\]
\end{lem}
\begin{proof}
Suppose that $\PR_T(x)$ satisfies $\EU$. 
For each $i_0 \leq k$, let $\varphi_{i_0}(\vec{z}, \vec{x})$ be the formula 
    \[
    \bigvee_{\substack{i \leq k \\ i \neq i_0}} \bigwedge_{j \leq l} (z_{i,j} = x_{i,j}).
    \]
    Since $T \vdash \bigwedge_{j \leq l} (z_{i_0, j} = z_{i_0, j})$, we have
    \[
        T \vdash \bigwedge_{j \leq l} (z_{i_0, j} = z_{i_0, j}) \lor \varphi_{i_0}(\vec{z}, \vec{x}).
    \]
    By Proposition \ref{EU_to_D1U}.1, $\PR_T(x)$ satisfies $\DU{1}$.
    So, 
    \[
        S \vdash \PR_T\Bigl(\gn{\bigwedge_{j \leq l} (\dot{z_{i_0, j}} = \dot{z_{i_0, j}}) \lor \varphi_{i_0}(\vec{\dot{z}}, \vec{\dot{x}}) }\Bigr).
    \]
    Let $v_0, \ldots, v_{i_0}$ be fresh variables.
    By the equality axiom,  
    \begin{align*}
        \PA \vdash \bigwedge 
        _{j \leq l}(z_{i_0, j} = v_j) & \land \PR_T\Bigl(\gn{\bigwedge_{j \leq l} (\dot{z_{i_0, j}} = \dot{z_{i_0, j}}) \lor \varphi_{i_0}(\vec{\dot{z}}, \vec{\dot{x}}) }\Bigr)\\
        & \quad \quad \to \PR_T\Bigl(\gn{\bigwedge_{j \leq l} (\dot{z_{i_0, j}} = \dot{v_j}) \lor \varphi_{i_0}(\vec{\dot{z}}, \vec{\dot{x}}) }\Bigr).
    \end{align*}
    Then, we obtain
    \[
        S \vdash \bigwedge 
        _{j \leq l}(z_{i_0, j} = v_j) \to \PR_T\Bigl(\gn{\bigwedge_{j \leq l} (\dot{z_{i_0, j}} = \dot{v_j}) \lor \varphi_{i_0}(\vec{\dot{z}}, \vec{\dot{x}}) }\Bigr).
    \]
    By substituting $x_{i_0, j}$ for $v_j$, we get
    \begin{equation}\label{eq1}
        S \vdash \bigwedge 
        _{j \leq l}(z_{i_0, j} = x_{i_0, j}) \to \PR_T\Bigl(\gn{\bigwedge_{j \leq l} (\dot{z_{i_0, j}} = \dot{x_{i_0, j}}) \lor \varphi_{i_0}(\vec{\dot{z}}, \vec{\dot{x}}) }\Bigr).
    \end{equation}
    Here, since
    \[
        T \vdash \bigwedge_{j \leq l} (z_{i_0, j} = x_{i_0, j}) \lor \varphi_{i_0}(\vec{z}, \vec{x}) \leftrightarrow \bigvee_{i \leq k} \bigwedge_{j \leq l} (z_{i,j} = x_{i,j}), 
    \]
    by applying $\EU$, we obtain
    \[
        S \vdash  \PR_T\Bigl(\gn{\bigwedge_{j \leq l} (\dot{z_{i_0, j}} = \dot{x_{i_0, j}}) \lor \varphi_{i_0}(\vec{\dot{z}}, \vec{\dot{x}}) }\Bigr) \leftrightarrow \PR_T\Bigl(\gn{\bigvee_{i \leq k} \bigwedge_{j \leq l} (\dot{z_{i,j}} = \dot{x_{i,j}})}\Bigr). 
    \]
    By combining this with (\ref{eq1}), we have
    \[
            S \vdash \bigwedge 
        _{j \leq l}(z_{i_0, j} = x_{i_0, j}) \to \PR_T\Bigl(\gn{\bigvee_{i \leq k} \bigwedge_{j \leq l} (\dot{z_{i,j}} = \dot{x_{i,j}})}\Bigr).
    \]
    Therefore we conclude
\[
    S \vdash \bigvee_{i \leq k} \bigwedge_{j \leq l} (z_{i,j} = x_{i,j}) \to \PR_T\Bigl(\gn{\bigvee_{i \leq k} \bigwedge_{j \leq l} (\dot{z_{i,j}} = \dot{x_{i,j}})}\Bigr).
\]
\end{proof}

Here, we recall the following fact. 

\begin{fact}[{\cite[Lemma 3.2]{Kura20}}\footnote{Unfortunately, the statement of \cite[Lemma 3.2]{Kura20} contains a typo.}]\label{subst}
For any formula $\varphi(\vec{y}, v)$, 
\[
    \PA \vdash \gn{\varphi(\vec{\dot{y}}, \dot{v})} [\s(x) \slash v] = \gn{\varphi(\vec{\dot{y}}, \s(\dot{x}))},
\]
where $\gn{\varphi(\vec{\dot{y}}, \dot{v})} [\s(x) \slash v]$ is the result of substituting $\s(x)$ for $v$ in $\gn{\varphi(\vec{\dot{y}}, \dot{v})}$. 
\end{fact}

For each term $t$, let $d(t)$ be the number of function and constant symbols contained in $t$. 
We say that a term $t$ is \textit{simple} iff $d(t) \leq 1$.

\begin{lem}\label{Lem2}
Suppose that $\PR_T(x)$ satisfies $\EU$. 
Then, for any family $\{z_{i,j}\}_{\substack{i \leq k \\ j \leq l}}$ of variables and any family $\{t_{i,j}(\vec{x})\}_{\substack{i \leq k \\ j \leq l}}$ of simple terms, we have
\[
    S \vdash \bigvee_{i \leq k} \bigwedge_{j \leq l} (z_{i,j} = t_{i,j}(\vec{x})) \to \PR_T\Bigl(\gn{\bigvee_{i \leq k} \bigwedge_{j \leq l} (\dot{z_{i,j}} = t_{i,j}(\vec{\dot{x}}))}\Bigr).
\]
\end{lem}
\begin{proof}
    We prove this lemma by induction on the number $p$ of members $t$ of the family $\{t_{i,j}(\vec{x})\}_{\substack{i \leq k \\ j \leq l}}$ with $d(t) = 1$. 

    If $p = 0$, then every $t_{i, j}$ is a variable, so the lemma follows from Lemma \ref{Lem1}. 

    Assume that the lemma holds for any family of simple terms in which $p=m$.
    Let $\{t_{i,j}(\vec{x})\}_{\substack{i \leq k \\ j \leq l}}$ be a family of simple terms such that $p=m+1$.
    By an appropriate equivalence transformation and in view of the application of the condition $\EU$, we may, without loss of generality, assume that $d(t_{k, l}(\vec{x})) = 1$.

    Let $\varphi(\vec{z}, \vec{x})$ and $\psi(\vec{z}, \vec{x})$ be formulas
    \begin{align*}
        \varphi(\vec{z}, \vec{x}) & : \equiv \bigvee_{i < k} \bigwedge_{j \leq l} (z_{i, j} = t_{i, j}(\vec{x})), \\
        \psi(\vec{z}, \vec{x}) & : \equiv \bigwedge_{j < l} (z_{k, j} = t_{k, j}(\vec{x})). 
    \end{align*}
    Then, $\bigvee_{i \leq k} \bigwedge_{j \leq l} (z_{i,j} = t_{i,j}(\vec{x}))$ is exactly $\varphi(\vec{z}, \vec{x}) \land (\psi(\vec{z}, \vec{x}) \land z_{k, l} = t_{k, l}(\vec{x}))$. 
    Since $d(t_{k, l}(\vec{x})) = 1$, we have that $t_{k, l}(\vec{x})$ is one of $0$, $\s(x)$, $x + y$, and $x \times y$. 
    
    Let $v$ be any fresh variable. 
    By the induction hypothesis,
    \begin{equation}\label{eq2}
    S \vdash \varphi(\vec{z}, \vec{x}) \land (\psi(\vec{z}, \vec{x}) \land z_{k, l} = v) \to \PR_T(\gn{\varphi(\vec{\dot{z}}, \vec{\dot{x}}) \land (\psi(\vec{\dot{z}}, \vec{\dot{x}}) \land \dot{z_{k, l}} = \dot{v})}).
    \end{equation}
    We prove
    \begin{align*}
    S \vdash \varphi(\vec{z}, \vec{x}) & \land (\psi(\vec{z}, \vec{x}) \land z_{k, l} = t_{k, l}(\vec{x}))\\
    & \to \PR_T(\gn{\varphi(\vec{\dot{z}}, \vec{\dot{x}}) \land (\psi(\vec{\dot{z}}, \vec{\dot{x}}) \land \dot{z_{k, l}} = t_{k, l}(\vec{\dot{x}}))})
    \end{align*}
    by distinguishing the following four cases. 

    \medskip

    Case 1: $t_{k, l}(\vec{x})$ is $0$. \\
    By substituting $0$ for $v$ in (\ref{eq2}), since $0$ is a numeral, we obtain
    \[
    S \vdash \varphi(\vec{z}, \vec{x}) \land (\psi(\vec{z}, \vec{x}) \land z_{k, l} = 0) \to \PR_T(\gn{\varphi(\vec{\dot{z}}, \vec{\dot{x}}) \land (\psi(\vec{\dot{z}}, \vec{\dot{x}}) \land \dot{z_{k, l}} = 0)}).
    \]

    \medskip

    Case 2: $t_{k, l}(\vec{x})$ is $\s(x)$. \\
    By substituting $\s(x)$ for $v$ in (\ref{eq2}), by Lemma \ref{subst}, we obtain
    \[
    S \vdash \varphi(\vec{z}, \vec{x}) \land (\psi(\vec{z}, \vec{x}) \land z_{k, l} = \s(x)) \to \PR_T(\gn{\varphi(\vec{\dot{z}}, \vec{\dot{x}}) \land (\psi(\vec{\dot{z}}, \vec{\dot{x}}) \land \dot{z_{k, l}} = \s(\dot{x}))}).
    \]

    \medskip

    Case 3: $t_{k, l}(\vec{x})$ is $x + y$. \\
    Let $u_0, u_1$ be fresh variables. 
    Let $\xi(u_0, u_1)$ be the formula
    \begin{align*}
        \varphi(\vec{z}, \vec{x}) & \land (\psi(\vec{z}, \vec{x}) \land z_{k, l} = u_0+u_1)\\
    & \to \PR_T(\gn{\varphi(\vec{\dot{z}}, \vec{\dot{x}}) \land (\psi(\vec{\dot{z}}, \vec{\dot{x}}) \land \dot{z_{k, l}} = \dot{u_0} + \dot{u_1})}).
    \end{align*}
    We prove $S \vdash \forall u_1\, \forall u_0 \, \xi(u_0, u_1)$ by using the induction axiom for $\forall u_0 \, \xi(u_0, u_1)$ on $u_1$. 
    By substituting $u_0$ for $v$ in (\ref{eq2}), considering $S \vdash z_{k, l} = u_0 \leftrightarrow z_{k,l} = u_0 + 0$ and applying $\EU$, we obtain $S \vdash \forall u_0\, \xi(u_0, 0)$\footnote{More precisely, we apply $\EU$ to the equivalence between $\varphi(\vec{z}, \vec{x}) \land (\psi(\vec{z}, \vec{x}) \land z_{k, l} = u_0)$ and $\varphi(\vec{z}, \vec{x}) \land (\psi(\vec{z}, \vec{x}) \land z_{k, l} = u_0 + 0)$.}.

    By the instantiation axiom $\forall u_0\, \xi(u_0, u_1) \to \xi(\s(u_0), u_1)$ and Lemma \ref{subst}, 
    \begin{align*}
        S \vdash \forall u_0\, \xi(u_0, u_1) \land \Bigl[\varphi(\vec{z}, \vec{x}) & \land (\psi(\vec{z}, \vec{x}) \land z_{k, l} = \s(u_0)+u_1)\\
    & \to \PR_T(\gn{\varphi(\vec{\dot{z}}, \vec{\dot{x}}) \land (\psi(\vec{\dot{z}}, \vec{\dot{x}}) \land \dot{z_{k, l}} = \s(\dot{u_0}) + \dot{u_1})}) \Bigr].
    \end{align*}
    Since $S \vdash z_{k,l} = \s(u_0) + u_1 \leftrightarrow z_{k, l} = u_0 + \s(u_1)$, by applying $\EU$, we get
    \begin{align*}
        S \vdash \forall u_0\, \xi(u_0, u_1) \land \Bigl[\varphi(\vec{z}, \vec{x}) & \land (\psi(\vec{z}, \vec{x}) \land z_{k, l} = u_0+\s(u_1))\\
    & \to \PR_T(\gn{\varphi(\vec{\dot{z}}, \vec{\dot{x}}) \land (\psi(\vec{\dot{z}}, \vec{\dot{x}}) \land \dot{z_{k, l}} = \dot{u_0} + \s(\dot{u_1}))}) \Bigr].
    \end{align*}
    This means $S \vdash \forall u_0\, \xi(u_0, u_1) \to \forall u_0\, \xi(u_0, \s(u_1))$, and thus $S \vdash \forall u_1\, \forall u_0\, \xi(u_0, u_1)$. 
    We conclude $S \vdash \xi(x, y)$. 

    \medskip

    Case 4: $t_{k, l}(\vec{x})$ is $x \times y$. \\
    Let $u_0, u_1$ be fresh variables. 
    Let $\eta(u_0, u_1)$ be the formula
    \begin{align*}
        \varphi(\vec{z}, \vec{x}) & \land (\psi(\vec{z}, \vec{x}) \land z_{k, l} = (x \times u_0)+u_1)\\
    & \to \PR_T(\gn{\varphi(\vec{\dot{z}}, \vec{\dot{x}}) \land (\psi(\vec{\dot{z}}, \vec{\dot{x}}) \land \dot{z_{k, l}} = (\dot{x} \times \dot{u_0}) + \dot{u_1})}).
    \end{align*}
    Since $S \vdash z_{k, l} = x \times y \leftrightarrow z_{k, l} = (x \times y) + 0$, by applying $\EU$, it suffices to prove $S \vdash \eta(y, 0)$. 
    To this end, we prove $S \vdash \forall u_0\, \forall u_1\, \eta(u_0, u_1)$ by using the induction axiom for $\forall u_1\, \eta(u_0, u_1)$ on $u_0$. 

    By substituting $u_1$ for $v$ in (\ref{eq2}), considering $S \vdash z_{k, l} = u_1 \leftrightarrow z_{k,l} = (x \times 0) + u_1$ and applying $\EU$, we obtain $S \vdash \forall u_1\, \eta(0, u_1)$.

    Let $u_2$ be a fresh variable and let $\rho(u_1, u_2)$ be the formula
    \begin{align*}
        \varphi(\vec{z}, \vec{x}) & \land (\psi(\vec{z}, \vec{x}) \land z_{k, l} = (x \times u_0)+(u_1 + u_2))\\
    & \to \PR_T(\gn{\varphi(\vec{\dot{z}}, \vec{\dot{x}}) \land (\psi(\vec{\dot{z}}, \vec{\dot{x}}) \land \dot{z_{k, l}} = (\dot{x} \times \dot{u_0}) + (\dot{u_1} + \dot{u_2})}).
    \end{align*}
    In the same way as in Case 3, it is proved that $S \vdash \forall u_1\, \eta(u_0, u_1) \to \forall u_1\, \rho(u_1, 0)$ and $S \vdash \forall u_1\, \rho(u_1, u_2) \to \forall u_1\, \rho(u_1, \s(u_2))$. 
    By the induction axiom, we get $S \vdash \forall u_1\, \eta(u_0, u_1) \to \forall u_2\, \forall u_1\, \rho(u_1, u_2)$, and hence $S \vdash \forall u_1\, \eta(u_0, u_1) \to \rho(u_1, x)$. 
    It follows that $S$ proves
    \begin{align*}
        \forall u_1\, \eta(u_0, u_1) \land \varphi(\vec{z}, \vec{x}) & \land (\psi(\vec{z}, \vec{x}) \land z_{k, l} = (x \times u_0)+(u_1 + x))\\
    & \to \PR_T(\gn{\varphi(\vec{\dot{z}}, \vec{\dot{x}}) \land (\psi(\vec{\dot{z}}, \vec{\dot{x}}) \land \dot{z_{k, l}} = (\dot{x} \times \dot{u_0}) + (\dot{u_1} + \dot{x})}).
    \end{align*}
    Since $S \vdash z_{k, l} = (x \times u_0) + (u_1 + x) \leftrightarrow z_{k, l} = (x \times \s(u_0)) + u_1$, by applying $\EU$, we have $S \vdash \forall u_1\, \eta(u_0, u_1) \to \forall u_1\, \eta(\s(u_0), u_1)$. 
    Therefore, we conclude $S \vdash \forall u_0\, \forall u_1\, \eta(u_0, u_1)$. 
\end{proof}

\begin{fact}
    For any quantifier-free formula $\delta(\vec{x})$, there exist natural numbers $k$ and $l$, a sequence $\vec{y}$ of variables, a family $\{z_{i,j}\}_{\substack{i \leq k \\  j \leq l}} \subseteq \vec{x} \cup \vec{y}$ of variables, and a family $\{t_{i,j}(\vec{x}, \vec{y})\}_{\substack{i \leq k \\  j \leq l}}$ of simple terms such that
    \[
        \PA \vdash \delta(\vec{x}) \leftrightarrow \exists \vec{y}\, \Bigl[\bigvee_{i \leq k} \bigwedge_{j \leq l} (z_{i,j} = t_{i,j}(\vec{x}, \vec{y})) \Bigr].
    \]
\end{fact}
\begin{proof}
Let $\delta(\vec{x})$ be any quantifier-free formula. 
First, replace every subformula of $\delta(\vec{x})$ that is of the form $t_0 \neq t_1$ with $t_0 \nleq t_1 \lor t_1 \nleq t_0$.
Then, replace subformulas of the form $t_0 \leq t_1$ with $\exists z\, (t_0 + z = t_1)$, and subformulas of the form $t_0 \nleq t_1$ with $\exists z\, (t_1 + \s(z) = t_0)$. 
Next, transform the resulting formula so that all terms become simple by eliminating nested occurrences of function symbols through the use of existential quantifiers. 
Then, replace every subformula that is of the form $t_0 = t_1$ with $\exists z\, (z = t_0 \land  z = t_1)$.
Then, we find a quantifier-free formula $\delta_0(\vec{x}, \vec{v})$ whose atomic subformulas are of the form $z = t(\vec{x}, \vec{y})$ for some simple term $t(\vec{x}, \vec{y})$ such that $\PA \vdash \delta(\vec{x}) \leftrightarrow \exists \vec{v}\, \delta_0(\vec{x}, \vec{v})$. 
Finally, rewrite $\delta_0(\vec{x}, \vec{v})$ in disjunctive normal form, ensuring in particular that each disjunct contains the same number of conjuncts.
\end{proof}

\begin{fact}[MRDP Theorem]
    For any $\Sigma_1$ formula $\varphi(\vec{x})$, there exists a quantifier-free formula $\delta(\vec{x}, \vec{y})$ such that
    \[
        \PA \vdash \varphi(\vec{x}) \leftrightarrow \exists \vec{y}\, \delta(\vec{x}, \vec{y}).
    \]
\end{fact}

\begin{cor}\label{Cor_MRDP}
    For any $\Sigma_1$ formula $\sigma(\vec{x})$, there exist natural numbers $k$ and $l$, a sequence $\vec{y}$ of variables, a family $\{z_{i,j}\}_{\substack{i \leq k \\  j \leq l}} \subseteq \vec{x} \cup \vec{y}$ of variables, and a family $\{t_{i,j}(\vec{x}, \vec{y})\}_{\substack{i \leq k \\  j \leq l}}$ of simple terms such that
    \[
        \PA \vdash \sigma(\vec{x}) \leftrightarrow \exists \vec{y}\, \Bigl[\bigvee_{i \leq k} \bigwedge_{j \leq l} (z_{i,j} = t_{i,j}(\vec{x}, \vec{y})) \Bigr].
    \]
\end{cor}

We are ready to prove Theorem \ref{EU_to_SCU}. 

\begin{proof}[Proof of Theorem \ref{EU_to_SCU}]
Suppose $\PR_T(x)$ satisfies $\EU$ and $\CBE$. 
Let $\sigma$ be any $\Sigma_1$ sentence. 
By Corollary \ref{Cor_MRDP}, there exist natural numbers $k$ and $l$, a sequence $\vec{y}$ of variables, a family $\{z_{i,j}\}_{\substack{i \leq k \\  j \leq l}} \subseteq \vec{y}$ of variables, and a family $\{t_{i,j}(\vec{y})\}_{\substack{i \leq k \\  j \leq l}}$ of simple terms such that
    \begin{equation}\label{eq5}
        \PA \vdash \sigma \leftrightarrow \exists \vec{y}\, \Bigl[\bigvee_{i \leq k} \bigwedge_{j \leq l} (z_{i,j} = t_{i,j}(\vec{y})) \Bigr].
    \end{equation}
    By applying $\E$, we get
    \begin{equation}\label{eq6}
        S \vdash \PR_T(\gn{\sigma}) \leftrightarrow \PR_T\Bigl(\gn{\exists \vec{y}\, \Bigl[\bigvee_{i \leq k} \bigwedge_{j \leq l} (\dot{z_{i,j}} = t_{i,j}(\vec{y})) \Bigr] } \Bigr).
    \end{equation}
    By Lemma \ref{Lem2}, we obtain
    \[
        S \vdash \bigvee_{i \leq k} \bigwedge_{j \leq l} (z_{i,j} = t_{i,j}(\vec{y})) \to \PR_T\Bigl(\gn{\bigvee_{i \leq k} \bigwedge_{j \leq l} (\dot{z_{i,j}} = t_{i,j}(\vec{\dot{y}})) } \Bigr).
    \]
    By combining this with (\ref{eq5}), we have
    \[
        S \vdash \sigma \to \exists \vec{y}\, \PR_T\Bigl(\gn{\bigvee_{i \leq k} \bigwedge_{j \leq l} (\dot{z_{i,j}} = t_{i,j}(\vec{\dot{y}})) } \Bigr).
    \]
    The application $\CBE$, together with (\ref{eq6}), yields
    \[
        S \vdash \sigma \to \PR_T(\gn{\sigma}). 
    \]
    This completes the proof. 
\end{proof}

Of course, by replacing $\CBE$ with $\CBEP$ in the above proof of Theorem \ref{EU_to_SCU}, we obtain a proof of ``$\EU$ \& $\CBEP \Rightarrow \SCU$'' which is exactly Fact \ref{uniform}.4.

By combining Theorem \ref{EU_to_SCU}, Propositions \ref{CB_local} and the fact ``$\D{2} \Leftrightarrow \M$ \& $\C$'', we obtain the following corollary. 

\begin{cor}
    $\EU$ \& $\CBE$ \& $\C \Rightarrow \D{2}$ \& $\D{3}$. 
\end{cor}

Thus, the set $\{\EU, \CBE, \C\}$ is sufficient for G2 of $\ConL_T$. 

Unfortunately, we do not know whether $\{\EU, \CBE\}$ is strictly weaker than $\{\EU, \CBEP\}$.

\begin{prob}
    Is there a $\Sigma_1$ provability predicate satisfying $\EU$ and $\CBE$, but not $\CBEP$?
\end{prob}

\section{Optimality}

In this section, we discuss the optimality of several versions of G2. 
To this end, we list some known results on the existence of provability predicates satisfying several conditions. 

\begin{fact}\label{examples}
    \begin{enumerate}
         \item \textup{(Arai \cite{Ara90})} There exists a $\Sigma_1$ provability predicate satisfying $\DG{3}$ such that $T \vdash \ConH_T$. 
        \item \textup{(\cite[Theorem 9]{Kura21})} There exists a $\Sigma_1$ provability predicate satisfying $\D{2}^{\mathrm{U}}$ and $\DCU$ such that $T \vdash \ConH_T$. 
        \item \textup{(\cite[Theorem 10]{Kura21})} When $S = T$, there exists a $\Sigma_1$ provability predicate satisfying $\CB$, $\D{2}$, and $\DCU$ such that $T \vdash \ConS_T$. 
        \item \textup{(\cite[Theorem 11]{Kura21})} When $S = T$, there exists a $\Sigma_1$ provability predicate satisfying $\CB$, $\M$, $\DU{3}$, $\DCU$, and $\Ros$.
        \item \textup{(Mostowski \cite[p.~24]{Mos65})} There exists a $\Sigma_1$ provability predicate satisfying $\DU{1}$, $\SCU$, and $\C$ such that $T \vdash \ConL_T$. 
    \end{enumerate}
\end{fact}

First, the optimality of ``$\E$ \& $\C$ \& $\D{3} \Rightarrow T \nvdash \ConL_T$'' (Corollary \ref{G2_refine_EC}.1 and Corollary \ref{G2_EC3}) is guaranteed by the following: 

\begin{itemize}
    \item $\{\E, \C\}$ is not sufficient for G2 of $\ConH_T$ by Fact \ref{examples}.2. 

    \item $\{\E, \D{3}\}$ is not sufficient for G2 of $\Ros$ by Fact \ref{examples}.4. 

    \item $\{\C, \D{3}\}$ is not sufficient for G2 of $\ConL_T$ by Fact \ref{examples}.5. 

    \item $\{\E, \C, \D{3}\}$ is strictly weaker than $\{\D{2}, \D{3}\}$ by the following Theorem \ref{ex1}. 

\end{itemize}

\begin{thm}\label{ex1}
There exists a $\Sigma_1$ provability predicate satisfying $\E$, $\C$, and $\D{3}$, but not $\M$. 
\end{thm}
\begin{proof}
We fix a $\Sigma_1$ provability predicate $\exists x\, \Proof_T(x, y)$ of $T$, where $\Proof_T(x, y)$ is a $\Delta_1(\PA)$ proof predicate $\Proof_T(x, y)$ of $T$ saying that ``$x$ is a $T$-proof of a formula $y$''. 
We may assume $\PA \vdash \Proof_T(x, y) \to y \leq x$. 
We fix a sentence $\xi$ that is not $\Sigma_1(T)$, that is, there is no $\Sigma_1$ sentence $\sigma$ such that $T \vdash \xi \leftrightarrow \sigma$ (cf.~\cite[Theorem 1.4]{Lin03}). 
For each number $m$, let $F_m$ be the set of all formulas whose G\"odel numbers are less than $m$. 

We define a provability predicate $\PR_T^\dagger(x)$ satisfying the requirements. 
In the definition of $\PR_T^\dagger(x)$, we use a family $\{X_{m, k}\}$ of finite set of formulas. 
By using the formalized Recursion Theorem, we may use $\PR_T^\dagger(x)$ itself in the definition of the family $\{X_{m, k}\}$. 

For each $m$, we recursively define the increasing sequence $X_{m, 0}, X_{m, 1}, \ldots$ of subsets of $F_m$ by referring to $\Proof_T(x, y)$ as follows:  

\begin{itemize}
    \item $X_{m, 0} = \{\varphi \in F_m \mid \varphi$ has a $T$-proof $p \leq m\}$. 

    \item Define $X_{m, k+1}$ to be the union of the following four sets: 

        \begin{enumerate}
            \item $X_{m, k}$, 
            \item $\{\varphi \in F_m \mid \exists \psi \in X_{m, k}$ s.t.~$\varphi \leftrightarrow \psi$ has a $T$-proof $p \leq m\}$, 
            \item $\{\varphi \land \psi \in F_m \mid \varphi, \psi \in X_{m, k}\}$, 
            \item $\{\PR_T^\dagger(\gn{\varphi}) \in F_m \mid \varphi \in X_{m, k}\}$. 
        \end{enumerate}
\end{itemize}
Let $X_m : = \bigcup_{k \in \omega} X_{m, k}$. 
The definition of $\{X_{m, k}\}$ can be carried out within $\PA$. 
Since $X_m \subseteq F_m$, it is easy to see that $X_m = X_{m, m}$ and so $x \in X_y$ is a $\Delta_1(\PA)$-definable binary relation. 
By using induction, it is proved that for $m' \geq m$ and $k' \geq k$, we have $X_{m,k} \subseteq X_{m', k'}$. 
These properties are also verified in $\PA$. 
Let $\PR^\dagger_T(x)$ be the formula
\[
    \exists y\, (x \in X_{y} \land \gn{\xi} \notin X_{y}). 
\]

For each natural number $m$, we prove that for all $k$, $(\varphi \in X_{m, k}$ \& $\xi \notin X_{m})$ implies $T \vdash \varphi$ by induction on $k$. 
If $(\varphi \in X_{m, 0}$ \& $\xi \notin X_{m})$, then $\varphi$ has a $T$-proof of $\varphi$, and so $T \vdash \varphi$.
Suppose that the statement holds for $k$. 
Assume $(\varphi \in X_{m, k+1}$ \& $\xi \notin X_{m})$ holds. 
We distinguish the following four cases: 
\begin{itemize}
    \item $\varphi \in X_{m, k}$: By the induction hypothesis, we obtain $T \vdash \varphi$. 

    \item there exists $\psi \in X_{m, k}$ such that $T \vdash \varphi \leftrightarrow \psi$: By the induction hypothesis, we have $T \vdash \psi$, and hence $T \vdash \varphi$. 

    \item $\varphi$ is of the form $\psi_0 \land \psi_1$ for some $\psi_0, \psi_1 \in X_{m, k}$: By the induction hypothesis, $T \vdash \psi_0$ and $T \vdash \psi_1$, and thus $T \vdash \varphi$. 

    \item $\varphi$ is of the form $\PR_T^\dagger(\gn{\psi})$ for some $\psi \in X_{m, k}$: Since $\xi \notin X_m$, we have that $\exists y\, (\gn{\psi} \in X_{y} \land \gn{\xi} \notin X_{y})$ ($\equiv \PR_T^\dagger(\gn{\psi})$) is a true $\Sigma_1$ sentence. 
    By $\Sigma_1$-completeness, we obtain $T \vdash \PR^\dagger_T(\gn{\psi})$. 
\end{itemize}
Since this argument can be carried out in $\PA$, we get 
\begin{equation}\label{eq100}
\PA \vdash \forall x\, (\PR^\dagger_T(x) \to \exists z\, \Proof_T(z, x)). 
\end{equation}
In a similar way, it is easily proved by induction on $k$ that for all $k$, $\varphi \in X_{m, k}$ implies that there exists a $\Sigma_1$ sentence $\sigma$ such that $T \vdash \varphi \leftrightarrow \sigma$. 
So, by the choice of $\xi$, we have that $\xi \notin X_m$ for all natural numbers $m$.

We show that $\PR^\dagger_T(x)$ is a provability predicate of $T$. 
Suppose $T \vdash \varphi$. 
Then, we find a $T$-proof $p$ of $\varphi$, and so $\varphi \in X_{p, 0} \subseteq X_p$. 
Since $\xi \notin X_p$, we have that $\PR_T^\dagger(\gn{\varphi})$ holds, and hence $S \vdash \PR_T^\dagger(\gn{\varphi})$ by $\Sigma_1$-completeness. 
If $S \vdash \PR_T^\dagger(\gn{\varphi})$, then $S \vdash \exists z\, \Proof_T(z, \gn{\varphi})$ by (\ref{eq100}). 
It follows that $T \vdash \varphi$. 

We prove that $\PR^\dagger_T(x)$ satisfies $\E$. 
Suppose $T \vdash \varphi \leftrightarrow \psi$. 
We find a standard $T$-proof $p$ of $\varphi \leftrightarrow \psi$. 
It suffices to prove $\PA \vdash \PR_T^\dagger(\gn{\varphi}) \to \PR_T^\dagger(\gn{\psi})$ because $\psi \leftrightarrow \varphi$ also has a standard $T$-proof. 
We reason in $\PA$: 
Suppose that $\PR_T^\dagger(\gn{\varphi})$ holds. 
Then, we find $m$ and $k$ such that $\varphi \in X_{m, k}$ and $\xi \notin X_m$. 
For $r: = \max \{m, p\}$, we have $\varphi \in X_{r, k}$ and $\xi \notin X_r$ because $p$ is standard. 
Then, $\psi \in X_{r, k+1}$. 
Hence $\PR_T^\dagger(\gn{\psi})$ holds. 

In the similar way, it is proved that $\PR_T^\dagger(x)$ also satisfies $\C$ and $\D{3}$. 

Finally, we prove that $\PR_T^\dagger(x)$ does not satisfy $\M$. 
Since $\PR_T^\dagger(x)$ satisfies $\E$, $\C$, and $\D{3}$, by Corollary \ref{G2_refine_EC}.1, we have $T \nvdash \neg\, \PR_T^\dagger(\gn{0=1})$. 
By the definition of $\PR^\dagger_T(x)$, we have $\PA \vdash \neg\, \PR_T^\dagger(\gn{\xi})$. 
Therefore, $T \nvdash \PR_T^\dagger(\gn{0=1}) \to \PR_T^\dagger(\gn{\xi})$. 
Since $T \vdash 0 = 1 \to \xi$, we conclude that $\PR_T^\dagger(x)$ does not satisfy $\M$. 
\end{proof}

We have already confirmed the optimality of ``$\C$ \& $\D{3} \Rightarrow$ non-$\Ros$'' (Theorem \ref{G2_C3}) and ``$\E$ \& $\D{3} \Rightarrow T \nvdash \ConS_T$'' (Theorem \ref{G2_refine_E}). 

Next, we consider the optimality of ``$\SC \Rightarrow T \nvdash \ConS_T$'' (a consequence of Theorem \ref{G2_refine_E}). 

\begin{itemize}
    \item $\D{3}$ is not sufficient for G2 of $\ConH_T$ by Fact \ref{examples}.1. 

    \item $\SC$ is not sufficient for G2 of $\Ros$ by the following Theorem \ref{ex2}. 
     
\end{itemize}

\begin{thm}\label{ex2}
When $S = T$, there exists a $\Sigma_1$ provability predicate satisfying $\SCU$ and $\Ros$, but not $\E$. 
\end{thm}
\begin{proof}
    Let $\PR_T(x)$ be a $\Sigma_1$ provability predicate satisfying $\E$ and $\Ros$, but not $\SC$ (e.g.~Fact \ref{examples}.2).
    Let $\PR^*_T(x)$ be the $\Sigma_1$ formula 
    \[
        (\neg \, \Sigma_1(x) \land \PR_T(x)) \lor (\Sigma_1(x) \land \True_{\Sigma_1}(x)),
    \]
    where $\Sigma_1(x)$ is a $\Delta_1(\PA)$ formula naturally expressing that ``$x$ is a $\Sigma_1$ sentence'' and $\True_{\Sigma_1}(x)$ is a $\Sigma_1$ formula saying that ``$x$ is a true $\Sigma_1$ sentence'' (cf.~\cite{HP93}). 
    The formula $\PR^*_T(x)$ obviously satisfies $\SCU$.
    
    For any formula $\varphi$, we have 
    \[
        \PR_T^*(\gn{\varphi}) \equiv \begin{cases} \varphi & \text{if}\ \varphi\ \text{is a}\ \Sigma_1\ \text{sentence}, \\
        \PR_T(\gn{\varphi}) & \text{otherwise.}\end{cases}
    \]
    This yields that $\PR^*_T(x)$ is a provability predicate of $T$ because $S = T$. 
    Also, it follows from $\Ros$ for $\PR_T(x)$ that $\PR_T^*(x)$ satisfies $\Ros$. 

    Since $\PR_T(x)$ does not satisfy $\SC$, we find a $\Sigma_1$ sentence $\sigma$ such that $T \nvdash \sigma \to \PR_T(\gn{\sigma})$. 
    Let $\chi$ be a sentence such that $T \vdash \chi$ and $\sigma \land \chi$ is not $\Sigma_1$. 
    Since $T \vdash \sigma \leftrightarrow (\sigma \land \chi)$, by applying $\E$ for $\PR_T(x)$, we have $T \vdash \PR_T(\gn{\sigma}) \leftrightarrow \PR_T(\gn{\sigma \land \chi})$. 
    So, $T \vdash \PR_T(\gn{\sigma}) \leftrightarrow \PR_T^*(\gn{\sigma \land \chi})$. 
    By the choice of $\sigma$, we have $T \nvdash \sigma \to \PR_T^*(\gn{\sigma \land \chi})$, and hence $T \nvdash \PR_T^*(\gn{\sigma}) \to \PR_T^*(\gn{\sigma \land \chi})$. 
    It follows that $\PR_T^*(x)$ does not satisfy $\E$. 
\end{proof}

At last, we consider ``$\CB$ \& $\DCU \Rightarrow T \nvdash \ConH_T$'' (Theorem \ref{HB_refine}). 

\begin{itemize}
    \item $\DCU$ is not sufficient for G2 of $\ConL_T$ by Fact \ref{examples}.2. 

    \item $\{\CB, \DCU\}$ is not sufficient for G2 of $\ConS_T$ by Fact \ref{examples}.3. 
     
\end{itemize}

However, we do not yet know whether this version of G2 is optimal, since the following problem remains unsolved.

\begin{prob}\label{prob_2}
\begin{enumerate}
    \item Is there a $\Sigma_1$ provability predicate satisfying $\CB$ such that $T \vdash \ConH_T$?

    \item Is there a $\Sigma_1$ provability predicate satisfying $\DU{1}$ and $\DCU$ such that $T \vdash \ConH_T$?
\end{enumerate}
\end{prob}

Problem \ref{prob_2}.1 is closely related a problem proposed in the previous work, which has not been settled yet.

\begin{prob}[{\cite[Problem 2.17]{Kura20}}]
Is there a $\Sigma_1$ provability predicate satisfying $\M$ and $\CB$ such that $T \vdash \ConH_T$?
\end{prob}

We recall that the following question also has not been settled yet.

\begin{prob}[{\cite[Problem 2.22]{Kura20}}]
Is there a $\Sigma_1$ provability predicate satisfying $\MU$ such that $T \vdash \ConL_T$?
\end{prob}

We propose the following relating problem. 

\begin{prob}\leavevmode
\begin{enumerate}
    \item Is there a $\Sigma_1$ provability predicate satisfying $\M$ and $\SC$ such that $T \vdash \ConL_T$?

    \item Is there a $\Sigma_1$ provability predicate satisfying $\E$ and $\SC$ such that $T \vdash \ConL_T$?
\end{enumerate}
\end{prob}

\section*{Acknowledgements}

The author would like to thank Haruka Kogure for continuing discussions at all stages of this study. 
This work was supported by JSPS KAKENHI Grant Number JP23K03200.

\bibliographystyle{plain}
\bibliography{ref}

\end{document}